\documentclass[a4paper, 12pt]{article}

\usepackage[margin=2cm]{geometry}
\usepackage{hyperref,xcolor}
\usepackage{amsthm, amssymb, amsmath}
\usepackage{enumerate}
\usepackage{algorithmic}
\usepackage{authblk}
\usepackage{graphics}
\usepackage{pstricks}
\usepackage{subfigure}

\newcommand{\Mod}[1]{\ (\mathrm{mod}\ #1)}
\newcommand{\F}{\mathcal{F}}
\newcommand{\M}{\mathcal{M}}
\newcommand{\bR}{\mathbb{R}}

\newcommand{\you}{\mathbf{u}}

\hypersetup{
    colorlinks=false,
	linkcolor=red,
	citecolor=red,
    pdfborder={0 0 0},
}
\date{}

\title{Fractional $L$-intersecting families}

\author[1]{Niranjan Balachandran}
\author[2]{Rogers~Mathew}
\author[3]{Tapas Kumar Mishra}

\affil[1]
{
	Department of Mathematics, \authorcr
	Indian Institute of Technology, Bombay 400076, India. \authorcr
	niranj@iitb.ac.in
}
\affil[2]
{
	Department of Computer Science and Engineering, \authorcr
	Indian Institute of Technology, Kharagpur 721302, India. \authorcr
 rogers@cse.iitkgp.ernet.in
}
\affil[3]
{
Department of Computer Science and Engineering, \authorcr
National Institute of Technology, Rourkela 769008, India. \authorcr
mishrat@nitrkl.ac.in
}

\theoremstyle{definition}
\newtheorem{definition}{Definition}
\newtheorem{example}[definition]{Example}

\theoremstyle{plain}
\newtheorem{theorem}{Theorem}
\newtheorem{lemma}[theorem]{Lemma}

\newtheorem{conjecture}[theorem]{Conjecture}
\newtheorem{openproblem}[theorem]{Open problem}

\theoremstyle{remark}

\newtheoremstyle{plainitshape}
  {}
  {}
  {\itshape}
  {}
  {\itshape}
  {.}
  {0.5em}
  {}
\theoremstyle{plainitshape}

\newtheorem{claim}{Claim}[theorem]

\newtheoremstyle{cases}
  {}
  {}
  {}
  {}
  {}
  {\newline}
  {0.5em}
  {{\itshape \thmname{#1}} \thmnumber{#2} ({\itshape\thmnote{#3}}).\medskip}
\theoremstyle{cases}

\newtheoremstyle{constructions}
  {}
  {}
  {}
  {}
  {}
  {}
  {0.5em}
  {{\itshape \thmname{#1}} \thmnumber{#2}\medskip}
\theoremstyle{constructions}

\definecolor{ao(english)}{rgb}{0.0, 0.5, 0.0}

\newcommand{\rogers}[1]{{\color{black} #1}}
\newcommand{\niranjan}[1]{{\color{black} #1}}
\newcommand{\tapas}[1]{{\color{black} #1}}

\newcommand{\Pa}{\mathcal{P}}
\newcommand{\ep}{\varepsilon}

\begin{document}
\maketitle
\begin{abstract}
Let $L = \{\frac{a_1}{b_1}, \ldots , \frac{a_s}{b_s}\}$, where for every $i \in [s]$, $\frac{a_i}{b_i} \in [0,1)$ is an irreducible fraction. Let $\mathcal{F} = \{A_1, \ldots , A_m\}$ be a family of  subsets of $[n]$. We say $\mathcal{F}$ is a \emph{fractional $L$-intersecting family} if for every distinct $i,j \in [m]$, there exists an $\frac{a}{b} \in L$ such that $|A_i \cap A_j| \in \{ \frac{a}{b}|A_i|, \frac{a}{b} |A_j|\}$. In this paper, we introduce and study the notion of fractional $L$-intersecting families.  

\end{abstract}
\section{Introduction}
\label{sec:Intro}
Let $[n]$ denote $\{1, \ldots , n\}$ and let $L = \{l_1, \ldots , l_s\}$ be a set of $s$ non-negative integers. A family $\mathcal{F} = \{A_1, \ldots , A_m\}$ of subsets of $[n]$ is \emph{$L$-intersecting} if for every $A_i, A_j \in \mathcal{F}$, $A_i \neq A_j$, $|A_i \cap A_j| \in L$. In 1975, it was shown by Ray-Chaudhuri and Wilson in \cite{ray1975t} that if $\mathcal{F}$ is $t$-uniform, then $|\mathcal{F}| \leq {n \choose s}$. Setting $L = \{0, \ldots ,s-1\}$, the family $\mathcal{F} = { [n] \choose s}$ is a tight example to the above bound, where ${[n] \choose s}$ denotes the set of all $s$-sized subsets of $[n]$. In the non-uniform case, it was shown by Frankl and Wilson in the year 1981 (see \cite{frankl1981intersection}) that if we don't put any restrictions on the cardinalities of the sets in $\mathcal{F}$, then $|F| \leq {n \choose s} + {n \choose s-1} + \cdots + {n \choose 0}$.  This  bound is tight as demonstrated by the set of all subsets of $[n]$ of size at most $s$ with $L = \{0, \ldots s-1 \}$. The proof of this bound was using the method of higher incidence matrices. Later, in 1991, Alon, Babai, and Suzuki in \cite{alon1991multilinear}  gave an elegant linear algebraic proof to this bound. They showed that if the cardinalities of the sets in $\mathcal{F}$ belong to the set of integers $K = \{k_1, \ldots , k_r\}$ with every $k_i > s-r$, then $|\mathcal{F}|$ is at most ${n \choose s} + {n \choose s-1} + \cdots + {n \choose s-r+1}$. The collection of all the subsets of $[n]$ of size at least $s-r+1$ and at most $s$ with $K=\{s-r+1, \ldots , s\}$ and $L=\{0, \ldots , s-1\}$ \rogers{forms} a tight example to this bound. In 2002, this result was extended by Grolmusz and Sudakov \cite{grolmusz2002k} to $k$-wise $L$-intersecting families.  
In 2003, Snevily showed in \cite{snevily2003sharp} that if $L$ is a collection of $s$ positive integers then $|\mathcal{F}| \leq {n-1 \choose s} + {n-1 \choose s-1} + \cdots + {n-1 \choose 0}$. See \cite{liu2014set} for a survey on $L$-intersecting families and their variants.  

In this paper, we introduce a new variant of $L$-intersecting families called the fractional $L$-intersecting families. 
Let $L = \{\frac{a_1}{b_1}, \ldots , \frac{a_s}{b_s}\}$, where for every $i \in [s]$, $\frac{a_i}{b_i} \in [0,1)$ is an irreducible fraction. Let $\mathcal{F} = \{A_1, \ldots , A_m\}$ be a family of  subsets of $[n]$.
We say $\mathcal{F}$ is a \emph{fractional $L$-intersecting family} if for every distinct $i,j \in [m]$, there exists an $\frac{a}
{b} \in L$ such that $|A_i \cap A_j| \in \{ \frac{a}{b}|A_i|, \frac{a}{b} |A_j|\}$. When $\mathcal{F}$ is $t$-uniform, it is an $L'$-intersecting family where $L' = \{ \lfloor \frac{a_1t}{b_1} \rfloor , \ldots , \lfloor \frac{a_st}{b_s}
\rfloor \}$ and therefore (using the result in \cite{ray1975t}), \rogers{$|\mathcal{F}| \leq {n \choose s}$}. A 
tight example to this bound is given by the family $\mathcal{F} = {[n] \choose t}$ where $L = \{\frac{0}{t}, \ldots , \frac{t-1}{t}\}$. So what is interesting is finding a good upper bound for $|\mathcal{F}|$ in the non-uniform case. Unlike in the case of the classical $L$-intersecting families, it is clear from the above definition that if $A$ and $B$ are two sets in a fractional $L$-intersecting family, then the cardinality of their intersection is a function of $|A|$ or $|B|$ (or both).   

In Section \ref{sec:main_thm_proof}, we prove the following theorem which gives an upper bound for the cardinality of a fractional $L$-intersecting family in the general case. We follow the convention that ${a \choose b}$ is $0$, when $b > a$. 
\begin{theorem}
\label{thm:main}
Let $n$ be a positive integer. Let $L = \{\frac{a_1}{b_1}, \ldots , \frac{a_s}{b_s}\}$, where for every $i \in [s]$, $\frac{a_i}{b_i} \in [0,1)$ is an irreducible fraction. Let $\mathcal{F}$ be a fractional $L$-intersecting family of subsets of $[n]$. Then, $|\mathcal{F}| \leq 2{n \choose s} g^2(t,n)\ln(g(t,n)) + \left(\sum_{i=1}^{s-1} {n \choose i}\right)g(t,n)$, where $g(t,n) = \frac{2(2t + \ln n)}{\ln(2t + \ln n)}$ and $t = \max(s,\max(b_i:i \in [s])~)$. Further, 
\begin{enumerate}
\item[(a)] if $s \leq n+1 - 2g(t,n)\ln(g(t,n))$, then $|\mathcal{F}| \leq 2{n \choose s} g^2(t,n)\ln(g(t,n))$, and 
\item[(b)] \rogers{if $t > n-c_1$, where $c_1$ is a positive integer constant, then $|\mathcal{F}| \leq 2c_1 {n \choose s}g(t,n)\ln(g(t,n))$ $ + c_1\sum_{i=1}^{s-1} {n \choose i}$.} 
\end{enumerate} 
\end{theorem}
Consider the following examples for a fractional $L$-intersecting family.
\begin{example} Let $L = \{\frac{0}{1}, \frac{1}{2}, \frac{1}{3}, \frac{2}{3}, \frac{1}{4}, \frac{3}{4},$ $\ldots , \frac{1}{n}, \ldots ,\frac{n-1}{n} \}$, where we omit fractions, like $\frac{2}{4}$, which are not irreducible. The collection of all the non-empty subsets of $[n]$ is a fractional $L$-intersecting family  of cardinality $2^n - 1$. Here, $|L| = s \in \Theta(n^2)$. Since $t \geq s$, we can apply Statement (b) of Theorem \ref{thm:main} to get an upper bound of $c_1 (2^n - 1)$ which is asymptotically tight. In general, when $L = \{\frac{0}{1}, \frac{1}{2}, \frac{1}{3}, \frac{2}{3}, \frac{1}{4}, \frac{3}{4},$ $\ldots , \frac{1}{n-c}, \ldots ,\frac{n-c-1}{n-c} \}$, where $c\geq 0$ is a constant, the set of all the non-empty subsets of $[n]$ of cardinality at most $n-c$ is \rogers{ an example which demonstrates that the bound given in Statement (b) of Theorem \ref{thm:main} is asymptotically tight}.    
\end{example}

\begin{example}
Let us now consider another example where $s~(= |L|)$ is a constant. Let $L = \{\frac{0}{s}, \frac{1}{s}, \ldots , \frac{s-1}{s}\}$. The collection of all the $s$-sized subsets of $[n]$ is a fractional $L$-intersecting family of cardinality ${n \choose s}$. In this case, the bound given by Theorem \ref{thm:main} is asymptotically tight up to a factor of $\frac{\ln^2 n}{\ln\ln n}$.  \rogers{We believe that if $\mathcal{F}$ is a fractional $L$-intersecting family of maximum cardinality, where $s$ ($=|L|$) is a constant, then $|\mathcal{F}| \in \Theta(n^s)$.}
\end{example}

\rogers{
Coming back to the classical $L$-intersecting families, it is known that when $\mathcal{F}$ is an $L$-intersecting family where $|L| = s = 1$, the Fisher's Inequality (see Theorem 7.5 in \cite{jukna2011extremal}) yields $|\mathcal{F}| \leq n$. Study of such intersecting families was initiated by Ronald Fisher in 1940 (see \cite{fisher1940examination}). This fundamental result of design theory is among the first results in the field of $L$-intersecting families. Analogously, consider the scenario when $L= \{\frac{a}{b}\}$ is a singleton set. Can we get a tighter (compared to Theorem \ref{thm:main}) bound in this case? We show in Theorem \ref{thm:singleton_b_is_a_prime} that if $b$ is a constant prime we do have a tighter bound. 
\begin{theorem}
\label{thm:singleton_b_is_a_prime}
\rogers{Let $n$ be a positive integer. Let $\mathcal{G}$ be a fractional $L$-intersecting families of subsets of $[n]$, where $L=\{\frac{a}{b}\}$, $\frac{a}{b} \in [0, 1)$, and $b$ is a prime. Then, $|\mathcal{G}| \leq (b-1)(n+1)\lceil \frac{\ln n}{\ln b} \rceil + 1$.} 
\end{theorem}
Assuming $L= \{ \frac{1}{2}\}$, Examples \ref{examp:bisectionClosedEasy} and \ref{examp:Hadamard} in Section \ref{sec:singleton} give  fractional $L$-intersecting families on $[n]$ of cardinality $\frac{3n}{2} - 2$ thereby implying that the bound obtained in Theorem \ref{thm:singleton_b_is_a_prime} is asymptotically tight up to a factor of $\ln n$ when $b$ is a constant prime.    
We believe that the cardinality of such families is at most $cn$, where $c>0$ is a constant.   
}  

The rest of the paper is organized in the following way: In Section \ref{sec:main_thm_proof}, we give the proof of Theorem \ref{thm:main} after introducing some necessary lemmas in the beginning. In Theorem \ref{thm:singleton_large_sets_in_F} in Section \ref{sec:LargeSets}, we give  an upper bound of $n$ for fractional $L$-intersecting families on $[n]$ whose member sets are `large enough'. In Section \ref{sec:singleton}, we consider the case when $L$ is a singleton set and give the proof of Theorem \ref{thm:singleton_b_is_a_prime}. Later in this section, in Theorem \ref{thm:bound_restricted_size_case}, we consider the case when the cardinalities of the sets in the fractional $L$-intersecting family are restricted. Finally, we conclude with some remarks, some open questions, and a conjecture.

\rogers{\section{The general case}}
\subsection{Proof of Theorem \ref{thm:main}}
\label{sec:main_thm_proof}
Before we move to the proof of Theorem \ref{thm:main}, we introduce a few lemmas that will be used in the proof. 
\subsubsection{Few auxiliary lemmas}
The following lemma is popularly known as the `Independence Criterion' or \rogers{`Triangular Criterion'}. 
\begin{lemma}[Lemma 13.11 in \cite{jukna2011extremal}, \rogers{Proposition 2.5} in \cite{babai1992linear}]
\label{lem:diag_criterion}
For $i = 1, \dots , m$ let $f_i:\Omega \rightarrow \mathbb{F}$ be functions and $v_i \in \Omega$ elements such that 
\begin{enumerate}
\item[(a)]$f_i(v_i) \neq 0$ for all $1 \leq i \leq m$;
\item [(b)]$f_i(v_j) = 0$ for all $1 \leq j < i \leq m$. 
\end{enumerate}
Then $f_1, \ldots , f_m$ are linearly independent members of the space $\mathbb{F}^{\Omega}$. 
\end{lemma}

\begin{lemma}
\label{lem:swallow-1}
Let $p$ be a prime; $\Omega = \{0,1\}^n$. Let $f \in \mathbb{F}_p^\Omega$ and let $i \in \mathbb{F}_p$. For any $A \subseteq [n]$, let $V_A \in \{0,1\}^n$ denote its $0$-$1$ incidence vector and let $x_A = \Pi_{j\in A} x_j$. Assume $f(V_A) \neq 0$, for every $|A| \not\equiv i~\Mod p$. Then, the set of functions $\{x_Af~:~|A| \not\equiv i~\Mod p \mbox{  and } |A| < p\}$ is linearly independent in the vector space $\mathbb{F}_p^{\{0,1\}^n}$ over $\mathbb{F}_p$. 
\end{lemma}
\begin{proof}
Arrange every subset of $[n]$ of cardinality less than $p$ in a linear order, say $\prec$, such that $A \prec B$ implies $|A| \leq |B|$. For any two distinct sets $A$ and $B$, we know that $x_A(V_B)f(V_B) = 0$ when $|B| \leq |A|$, where $x_A(V_B)$ denote the evaluation of the function $x_A$ at $V_B$. Suppose $\sum_{A:|A| \not\equiv i~\Mod p,~|A|<p}\lambda_Ax_Af = 0$ has a non-trivial solution. Then, identify the first set, say $A_0$, in  the linear order $\prec$ for which $\lambda_{A_0}$ is non-zero. Evaluate the functions on either side of the above equation at $V_{A_0}$ to get $\lambda_{A_0} = 0$ which  is a contradiction to our assumption.    
\end{proof}
The following lemma is from \cite{babai1992linear} (see Lemma 5.38).

\begin{lemma}[Lemma 5.38 in \cite{babai1992linear}]
\label{lem:swallow-2}
Let $p$ be a prime; $\Omega = \{0,1\}^n$. Let $f \in \mathbb{F}_p^\Omega$ be defined as $f(x) = \sum_{i=1}^n x_i - k$. For any $A \subseteq [n]$, let $V_A \in \{0,1\}^n$ denote its $0$-$1$ incidence vector and let $x_A = \Pi_{j\in A} x_j$. Assume $0\leq s,k \leq p-1$ and $s+k \leq  n$. Then, the set of functions $\{x_Af~:~|A| \leq s-1\}$ is linearly independent in the vector space $\mathbb{F}_p^\Omega$ over $\mathbb{F}_p$. 
\end{lemma}

\subsubsection{The proof}
\begin{proof}
Let $p$ be a prime and let $p > t$. We partition $\mathcal{F}$ into $p$ parts, namely $\mathcal{F}_0, \ldots , \mathcal{F}_{p-1}$, where $\mathcal{F}_i = \{ A \in \mathcal{F}~:~|A| \equiv i~ (mod~ p)\}$. 
\subsubsection*{Estimating $|\mathcal{F}_i|$, when $i>0$.}
Let $\mathcal{F}_i = \{A_1, \ldots , A_m\}$ and let $V_1, \ldots , V_m$ denote their corresponding $0$-$1$ incidence vectors. Define $m$ functions $f_1$ to $f_m$, where each $f_j \in \mathbb{F}_p^{\{0,1\}^n}$, in the following way.
$$ f_j(x) = (\langle V_j, x\rangle - \frac{a_1}{b_1}i)(\langle V_j, x\rangle - \frac{a_2}{b_2}i)\cdots (\langle V_j, x\rangle - \frac{a_s}{b_s}i). $$

Note that since $|A_j|\equiv i \Mod p$, $\left \langle V_j,V_j \right \rangle \equiv i \Mod p$.
Since $p > t$, for every $l \in [s]$, $i \not\equiv \frac{a_l}{b_l}i \Mod p$ unless $i \equiv 0 \Mod p$.
So,
\begin{equation}
\label{eqn:IndepCrit}
f_j(x) \begin{cases}
\neq &0 \text{, if $x = V_j$} \\
= &0 \text{,  otherwise}.
\end{cases}
\end{equation}

So, $f_j$'s are linearly independent in the vector space ${\mathbb{F}_p}^{\{0,1\}^n}$ over ${\mathbb{F}_p}$
(by Lemma \ref{lem:diag_criterion}). Since $x = (x_1, x_2, \ldots , x_n) \in \{0,1\}^n$, $x_i^r = x_i$, for any positive integer $r$. Each $f_j$ is thus an appropriate linear combination of distinct monomials of degree at most $s$.
Therefore, $|\mathcal{F}_i| = m \leq \sum_{j=0}^s {n \choose j}$. We can  improve this bound by using the ``swallowing trick'' in a way similar to the way it is used in the proof of Theorem 1.1 in \cite{alon1991multilinear}. Let $f \in \mathbb{F}_p^{\{0,1\}^n}$ be  defined as $f(x) = \sum_{j \in [n]}x_j - i$. From Lemma \ref{lem:swallow-1}, we know that the set of functions $\{x_Af~:~|A| \not \equiv i \Mod p \mbox{ and } |A| < s \}$ is linearly independent in the vector space  $\mathbb{F}_p^{\{0,1\}^n}$ over $\mathbb{F}_p$. 
\begin{claim}
\label{claim:swallow}
$\{f_j~:~1 \leq j \leq m\} \cup \{x_Af~:~|A| \not \equiv i \Mod p \mbox{ and } |A| < s \}$ is a collection of functions that is linearly independent in the vector space  $\mathbb{F}_p^{\{0,1\}^n}$ over $\mathbb{F}_p$. 
\end{claim}
In order to prove the claim, assume $\sum_{j=1}^m\lambda_jf_j + \sum_{A:|A| \leq s-1,~|A| \not \equiv i \Mod p}\mu_Ax_Af = 0$ for some $\lambda_j,\mu_A \in \mathbb{F}_p$. Evaluating at $V_j$, all terms in the second sum vanish \rogers{(since $f(V_j)=0$)} and by Equation \ref{eqn:IndepCrit}, only the  term with subscript $j$ remains of the first sum. We infer that $\lambda_j = 0$, for every $j$. It then follows from Lemma \ref{lem:swallow-1} that every $\mu_A$ is zero thus proving the claim. 

Since each function in the collection of functions in Claim \ref{claim:swallow} can be obtained as a linear  combination of distinct monomials of degree at most $s$, we can infer that $m + \sum_{j \neq i, j=0}^{s-1}~{n \choose j} \leq \sum_{j=0}^s {n \choose j}$. We thus have 

\begin{equation}
\label{eqn:general_bound_F_i}
|\mathcal{F}_i| \leq \begin{cases}
 & {n \choose s} + {n \choose i} \text{, if $i<s$} \\
 &{n \choose s} \text{, otherwise} \\
\end{cases}
\end{equation}

Observe that $ i\leq p-1$. We will shortly see that the prime $p$ we choose is always at most $ 2g(t,n)\ln(g(t,n))$, where $g(t,n) = \frac{(2t + \ln n)}{\ln(2t + \ln n)}$. So if $s \leq n+1 - 2g(t,n)\ln(g(t,n))$, the condition $s+i \leq n$ (here $i$ stands for  the symbol $k$ in Lemma \ref{lem:swallow-2}) given in Lemma \ref{lem:swallow-2} is satisfied and therefore the more powerful Lemma \ref{lem:swallow-2} can be used instead of Lemma \ref{lem:swallow-1} while applying  the swallowing trick.    We can then claim that (proof  of this claim is similar to the proof of Claim \ref{claim:swallow} and is therefore omitted) 
$\{f_j~:~1 \leq j \leq m\} \cup \{x_Af~:~ |A| < s \}$ (, where $f(x) = \sum_{j=1}^n x_j - i$) is a collection of functions that is linearly independent in the vector space  $\mathbb{F}_p^{\{0,1\}^n}$ over $\mathbb{F}_p$ which can be obtained as a linear combination of distinct monomials of  degree at most $s$. It then follows that $|\mathcal{F}_i| \leq {n \choose s}$. 

In the rest of the proof, we shall assume the general bound for $|\mathcal{F}_i|$ given by Inequality \ref{eqn:general_bound_F_i}. (Using the ${n \choose s}$ upper bound for $|\mathcal{F}_i|$ in place of Inequality \ref{eqn:general_bound_F_i} when $s \leq n+1 - 2g(t,n)\ln(g(t,n))$ in the rest of the proof will yield the tighter bound for $|\mathcal{F}|$ given in Statement (a) in the theorem.)   

Observe that we still do not have an estimate of $|\mathcal{A}_0|$ since  $i \equiv \frac{a_l}{b_l}i \Mod p$ when $i\equiv 0 \Mod p$. To overcome this problem, 
consider the collection $P=\{p_{q+1},\ldots ,p_r\}$ of $r-q$ smallest primes with $p_q \leq t < p_{q+1} < \cdots< p_r$ ($p_j$ denotes the $j$-th prime; $p_1=2, p_2=3,$ and so on) such that for every $A \in \mathcal{F}$, there exists a prime $p \in P$ with $p \nmid {|A|}$. Note that if we repeat the steps done above for each $p \in P$,
we obtain the following  upper bound.
\begin{eqnarray*}
|\mathcal{F}| &  \leq & (p_{q+1}+ \cdots +p_r - (r-q) ){n \choose s} + (r-q)\sum_{j=1}^{s-1}{n \choose j} \nonumber \\
&  < & (r-q)\left(p_r{n \choose s} + \sum_{j=1}^{s-1}{n \choose j}\right)\nonumber \\
\end{eqnarray*}

To obtain a small cardinality set $P$ of the desired requirement,
we choose the minimum $r$ such that $p_{q+1}p_{q+2} \cdots p_r > n$. 
\rogers{If $t > n-c_1$, for some positive integer constant $c_1$, then $P=\{p_{q+1}, \ldots , p_{q+c_1}\}$ satisfies the desired requirements of $P$.} We thus have, 
\begin{eqnarray}
|\mathcal{F}| &  < \begin{cases}
 & c_1 \left(p_r{n \choose s} + \sum_{j=1}^{s-1}{n \choose j}\right) \text{, \rogers{if $t > n - c_1$ (here $c_1$ is a positive integer constant)}} \\
 & r \left(p_r{n \choose s} + \sum_{j=1}^{s-1}{n \choose j}\right) \text{, otherwise}
 \end{cases}
\label{eqn:upperBoundF}
\end{eqnarray}

The product of the first $k$ primes is the \emph{primorial function} $p_k\#$ and it is known that 
$p_k\#= e^{(1+o(1))k\ln k}$. Given a natural number $N$, let $N\#$ denote the product of all the primes less than or equal to $N$ (some call this the primorial function). It is known that $N\# = e^{(1+o(1))N}$.  
Since $\frac{p_r\#}{t\#}= p_{k+1}p_{k+2} \cdots p_r$, setting $\frac{e^{(1+o(1))r\ln r}}{e^{(1+o(1))t}} > n$, we get, $r \leq  \frac{2(2t + \ln n)}{\ln(2t + \ln n)} = g(t,n)$. 
Using the prime number theorem, the $r$th prime $p_r$ is at most $2r \ln r$. 
Thus, we have $p_r \leq 2g(t,n)\ln(g(t,n))$. Substituting for $r$ and $p_r$ in Inequality \ref{eqn:upperBoundF} gives the theorem.
\end{proof}

\subsection{When the sets in $\mathcal{F}$ are `large enough'}
\label{sec:LargeSets}
In the following theorem, we show that when  the sets in a fractional $L$-intersecting $\mathcal{F}$ are `large enough', then $|\mathcal{F}|$ is at most $n$. 

\begin{theorem}
\label{thm:singleton_large_sets_in_F}
\rogers{Let $n$ be a positive integer. Let $L= \{\frac{a_1}{b_1}, \ldots , \frac{a_s}{b_s}\}$, where for every $i \in [s]$, $\frac{a_i}{b_i} \in [0,1)$ is an irreducible fraction. Let $\frac{a}{b} = \max(\frac{a_1}{b_1}, \ldots , \frac{a_s}{b_s})$. Let $\mathcal{F}$ be a fractional $L$-intersecting family of subsets of $[n]$ such that for every $A \in \mathcal{F}$, $|A| > \alpha n$, where $\alpha = \max(\frac{1}{2}, \frac{4a - b}{2b})$. Then, $|\mathcal{F}| \leq n$.} 
\end{theorem}
\begin{proof}
Let $\mathcal{F} = \{A_1, A_2, \ldots , A_m\}$. For every $A_i \in \mathcal{F}$, we define its $(+1,-1)$-incidence vector as:
 	\begin{align}
	X_{A_i}(j)= \begin{cases}
	+1, \text{ if $j \in A_i$} \\
	-1, \text{ if $j \not\in A_i$}. 
	\end{cases}
	\end{align}
We prove the theorem by proving the following claim. 
\begin{claim}
$X_{A_1}, \ldots , X_{A_m}$ are linearly independent in the vector space $\mathbb{R}^n$ over $\mathbb{R}$. 
\end{claim}
Assume for contradiction that $X_{A_1}, \ldots , X_{A_m}$ are linearly dependent in the vector space $\mathbb{R}^n$ over $\mathbb{R}$. Then, we have some reals $\lambda_{A_1}, \ldots , \lambda_{A_m}$ where not all of them are zeroes such that 
\begin{equation}
\label{eqn:linIndep}
\lambda_{A_1}X_{A_1} + \cdots + \lambda_{A_m}X_{A_m} = 0. 
\end{equation} 
It is given that, for every $A_i \in \mathcal{F}$, $|A_i| > \frac{n}{2}$. Let $u = (1,1, \ldots , 1) \in \mathbb{R}^n$ be the all ones vector. Then, $\langle X_{A_i}, u \rangle > 0$, for every $A_i \in \mathcal{F}$. Therefore, if all non-zero $\lambda_{A_i}$s in Equation (\ref{eqn:linIndep}) are of the same sign, say positive, then the inner product of $u$ with the L.H.S of Equation (\ref{eqn:linIndep}) would be non-zero which is a contradiction. Hence, we can assume that not all $\lambda_{A_i}$s are of the same sign. We rewrite Equation (\ref{eqn:linIndep}) by moving all negative $\lambda_{A_i}$s to the R.H.S. Without loss of generality, assume $\lambda_{A_1}, \ldots , \lambda_{A_k}$ are \rogers{non-negative} and the rest are negative. Thus, we have 
\rogers{\[ v = \lambda_{A_1}X_{A_1} + \cdots + \lambda_{A_k}X_{A_k} = - (\lambda_{A_{k+1}}X_{A_{k+1}} + \cdots + \lambda_{A_m}X_{A_m}), \]}
where $v$ is a non-zero vector. 

For any two distinct sets $A, B \in \mathcal{F}$, \rogers{$\exists \frac{a_i}{b_i} \in L$ such that  
\begin{align}
	\left \langle X_A,X_B\right \rangle = \begin{cases}
	n-2|A| + \frac{4a_i - 2b_i}{b_i}|B|, \text{ if $|A \cap B| = \frac{a_i}{b_i}|B|$, }\\
	n-2|B| + \frac{4a_i - 2b_i}{b_i}|A|, \text{ otherwise (that is, if $|A \cap B| = \frac{a_i}{b_i}|A|)$. }
	\end{cases}
	\end{align}
Since $\frac{a}{b} = \max(\frac{a_1}{b_1}, \ldots , \frac{a_s}{b_s})$, we have $\left \langle X_A,X_B\right \rangle \leq n-2|A| + \frac{4a - 2b}{b}|B|$ or $\left \langle X_A,X_B\right \rangle \leq n-2|B| + \frac{4a - 2b}{b}|A|$. 	
Applying the fact that the cardinality of every set $S$ in $\mathcal{F}$ satisfies $\alpha n < |S| \leq n$, where $\alpha = \max(\frac{1}{2}, \frac{4a - b}{2b})$, we get $\langle X_A,X_B \rangle < 0$. This implies that $\langle v, v \rangle = \langle \lambda_{A_1}X_{A_1} + \cdots + \lambda_{A_k}X_{A_k}  , -(\lambda_{A_{k+1}}X_{A_{k+1}} + \cdots + \lambda_{A_m}X_{A_m}) \rangle < 0$ which is a contradiction.} This proves the claim and thereby the theorem. 
\end{proof}
 
\section{$L$ is a singleton set}
\label{sec:singleton}
\rogers{As explained in Section \ref{sec:Intro}, the Fisher's Inequality is a special case of the classical $L$-intersecting families, where $|L| = 1$. In this section, we study fractional $L$-intersecting families with $|L| = 1$; a fractional variant of the Fisher's inequality. }
\subsection{Proof of Theorem \ref{thm:singleton_b_is_a_prime}}
\emph{Statement of Theorem \ref{thm:singleton_b_is_a_prime}:} Let $n$ be a positive integer. Let $\mathcal{G}$ be a fractional $L$-intersecting families of subsets of $[n]$, where $L=\{\frac{a}{b}\}$, $\frac{a}{b} \in [0, 1)$, and $b$ is a prime. Then, $|\mathcal{G}| \leq (b-1)(n+1)\lceil \frac{\ln n}{\ln b} \rceil + 1$.
\begin{proof}
\rogers{It is easy to see that if $a=0$, then $|\mathcal{G}| \leq n$ with the set of all singleton subsets of $[n]$ forming a tight example to this bound. So assume $a \neq 0$. Let $\mathcal{F} = \mathcal{G}\setminus \mathcal{H}$, where $\mathcal{H} = \{A \in \mathcal{G}~:~b \nmid |A|\}$. From the definition of a fractional $\frac{a}{b}$-intersecting family it is clear that $|\mathcal{H}|\leq 1$. The rest of the proof is to show that $|\mathcal{F}| \leq (b-1)(n+1)\lceil \frac{\ln n}{\ln b} \rceil$. We do this by partitioning $\mathcal{F}$ into $(b-1)\lceil \log_b n \rceil$ parts and then showing that each part is of size at most $n+1$.} We define $F_i^j$ as 
	\[\mathcal{F}_i^j= \{A \in \mathcal{F}| |A| \equiv j \Mod i\}.\]
Since $b$ divides $|A|$, for every $A \in \mathcal{F}$, under this definition $\mathcal{F}$ can be partitioned into families $\mathcal{F}_{b^k}^{ib^{k-1}}$, where $2 \leq k \leq \lceil \log_b n \rceil$ and $1 \leq i \leq b-1$. We show that, for every $i \in [b-1]$ and for every $2 \leq k \leq \lceil \log_b n\rceil$, $|\mathcal{F}_{b^k}^{ib^{k-1}}| \leq n+1$.       	

In order to estimate $|\mathcal{F}_{b^k}^{ib^{k-1}}|$, for each $A \in \mathcal{F}_{b^k}^{ib^{k-1}}$, create a vector $X_A$ as follows:

\[X_A(j)=\begin{cases}
	\frac{1}{\sqrt{b^{k-2}}} \text{, if $j \in A$;}\\
	0 \text{, otherwise.}
	\end{cases} \]
	\rogers{Note that, for $A,B \in \mathcal{F}_{b^k}^{ib^{k-1}}$}
	\begin{align}
	\left \langle X_A,X_B\right \rangle \equiv \begin{cases}
	b \Mod {b^2}, \text{ if $A=B$, }\\
	ai \Mod b, \text{ if $A\neq B$, }
	\end{cases}
	\end{align} 
Let $|\mathcal{F}_{b^k}^{ib^{k-1}}| = m$. Let $M_{k,i}$ denote the $m \times n$ matrix formed by taking $X_A$s as rows for each $A \in \mathcal{F}_{b^k}^{ib^{k-1}}$. Then, $|\mathcal{F}_{b^k}^{ib^{k-1}}| \leq n+1$ can be proved by considering $B = M_{k,i} \times M_{k,i}^T$ and showing that $B - aiJ$ (, where $J$ is the $m \times m$ all 1 matrix, ) has full rank; determinant of $B - aiJ$ is non-zero since the only term not divisible by the prime $b$ in the expansion of its determinant comes from the product of all the diagonals (note that $a<b$, $i<b$, and since $b$ is a prime, we have $b \nmid ai$).
\end{proof}
We shall call $\mathcal{F}$ a \emph{bisection closed family} if $\mathcal{F}$ is a fractional $L$-intersecting family where $L=\{\frac{1}{2}\}$. We have two different constructions of families that are bisection closed and are of cardinality $\frac{3n}{2}-2$ on $[n]$. 
\begin{example} 
\label{examp:bisectionClosedEasy}
\rogers{Let $n$ be an even positive integer. 
Let $\mathcal{B}$ denote the collection of 2-sized sets that contain only 1 as a common element in any two sets}, i.e. $\{1,2\}, \{1,3\}, \ldots, \{1,n\}$; and
\rogers{let  $\mathcal{C}$ denote collection of 4-sized sets that contain only $\{1,2\}$ as common elements}, i.e. $\{1,2,3,4\},$ $ \{1,2,5,6\}, \ldots, \{1,2,n-1,n\}$.
It is not hard to see that $\mathcal{B} \cup \mathcal{C}$ is indeed bisection closed.
\end{example}
\begin{example}
\label{examp:Hadamard}
The second example of a bisection closed family of cardinality $\frac{3n}{2}-2$ comes from \emph{Recursive Hadamard matrices}.
A Recursive Hadamard matrix $H(k)$ of size $2^k \times 2^k$ can be obtained from $H(k-1)$ of size $2^{k-1} \times 2^{k-1}$ as follows
\begin{align*}
H(k)= \left[ \begin{array}{cc}
H(k-1) & H(k-1) \\
H(k-1)  & -H(k-1)  \\
\end{array} \right],
\end{align*}
where $H(0)=1$.
Now consider the matrix:
\begin{align*}
M(k)= \left[ \begin{array}{cc}
H(k-1) & H(k-1) \\
H(k-1)  & -H(k-1)  \\
H(k-1) & J(k-1) \\
\end{array} \right], \text{ where $J(k-1)$ denotes the $2^{k-1}\times 2^{k-1}$ all 1s' matrix}.
\end{align*}
Let $M'(k)$ be the matrix obtained from $M(k)$ by removing the \rogers{first and  the }$(2^k+1)$th rows and  replacing the -1's by 1's and 1's by 0's.
$M'(k)$ is clearly bisection closed and has cardinality $\frac{3n}{2}-2$, where $n=2^k$. 
\end{example}
\subsection{Restricting the cardinalities of the sets in $\mathcal{F}$}
When $L = \{\frac{a}{b}\}$, where $b$ is a prime, Theorem \ref{thm:singleton_b_is_a_prime} yields an upper bound of $O(\frac{b}{\log b} n \log n)$ for $|\mathcal{F}|$. \rogers{However, we believe that when $|L|=1$, the cardinality of any fractional $L$-intersecting family on $[n]$ would be at most $cn$, where $c>0$ is a constant. To this end, we show in Theorem \ref{thm:bound_restricted_size_case} that when the sizes of the sets in $\mathcal{F}$ are restricted, we can achieve this.}

The following lemma is crucial to the proof of Theorem \ref{thm:bound_restricted_size_case}. 
\begin{lemma}\cite{alon2009perturbed,CODENOTTI200089}
	Let $A$ be an $m \times m$ real symmetric matrix with $a_{i,i}=1$ and $|a_{i,j}| \leq \epsilon$ for all $i \neq j$. 
	Let $tr(A)$ denote the trace of $A$, i.e., the sum of the diagonal entries of A. 
	Let $rk(A)$ denote the rank of $A$. Then,
	\[rk(A) \geq \frac{(tr(A))^2}{tr(A^2)}\geq \frac{m}{1+(m-1)\epsilon^2}.\]
	\label{lemma:alon}
\end{lemma}
\begin{proof}
	Let $\lambda_1, \ldots,\lambda_m$ denote the eigenvalues of $A$. Since only $rk(A)$ eigenvalues of $A$ are non-zero, $(tr(A))^2= (\sum_{i=1}^m \lambda_i)^2= (\sum_{i=1}^{rk(A)} \lambda_i)^2 \leq rk(A)\sum_{i=1}^{rk(A)} \lambda_i^2 = rk(A) tr(A^2)$, where the inequality follows from the Cauchy-Schwartz Inequality. Thus, $rk(A) \geq \frac{(tr(A))^2}{tr(A^2)}$. Substituting $tr(A) = m$ and $tr(A^2) = m + m(m-1)\epsilon^2$ in the above inequality proves the theorem.	
\end{proof}

\begin{theorem}
\label{thm:bound_restricted_size_case}
Let $n$ be a positive integer and let $\delta > 1$. Let $\mathcal{F}$ be a fractional $L$-intersecting family of subsets of $[n]$, where $L = \{\frac{a}{b}\}$, $\frac{a}{b} \in [0,1)$ is an irreducible fraction and for every $A \in \mathcal{F}$, $|A|$ in an integer in the range $\left[\frac{b}{4(b-a)}n - \frac{b}{4a\delta}\sqrt{n}, \frac{b}{4(b-a)}n + \frac{b}{4a\delta}\sqrt{n}\right]$. Then, $|\mathcal{F}| < \frac{\delta^2}{\delta^2 - 1}n$. 
\end{theorem}
\begin{proof}
For any $A \in \mathcal{F}$, let $Y_A \in \mathbb{R}^n$ be a vector defined as:
\begin{align*}
	Y_{A}(j)= \begin{cases}
	+\frac{1}{\sqrt{n}}, \text{ if $j \in A$} \\
	-\frac{1}{\sqrt{n}}, \text{ if $j \not\in A$}. 
	\end{cases}
	\end{align*}
Clearly, $\langle Y_A, Y_A \rangle = 1$. For any two distinct sets $A, B \in \mathcal{F}$, we have 
\begin{align}
\label{eqn:DotProd}
	\left \langle Y_A,Y_B\right \rangle = \begin{cases}
	\frac{n-2|A| + \frac{4a - 2b}{b}|B|}{n}, \text{ if $|A \cap B| = \frac{a}{b}|B|$, }\\
	\frac{n-2|B| + \frac{4a - 2b}{b}|A|}{n}, \text{ otherwise (that is, if $|A \cap B| = \frac{a}{b}|A|)$. }
	\end{cases}
	\end{align}
Suppose $\mathcal{F} = \{A_1, \ldots , A_m\}$. Let $B$ be the $m \times n$ matrix with $Y_{A_1}, \ldots , Y_{A_m}$ as its rows. Then, from Equation \ref{eqn:DotProd}, it follows that $BB^T$ is an $m \times m$ real symmetric matrix with the diagonal entries being $1$ and the absolute value of any other entry being at most $\frac{1}{\delta \sqrt{n}}$. Applying Lemma \ref{lemma:alon}, we have $n \geq rk(BB^T) \geq \frac{m}{1 + \frac{m-1}{\delta^2n}} > \frac{m}{1 + \frac{m}{\delta^2n}}$. Thus, $n + \frac{m}{\delta^2} > m$ or $m < \frac{\delta^2}{\delta^2 - 1}n$. 
\end{proof}

\section{Discussion}
\rogers{ In Theorem \ref{thm:main}, we gave a general upper bound for $|\mathcal{F}|$, where $\mathcal{F}$ is a fractional $L$-intersecting family. In Section \ref{sec:Intro}, we also gave an example to show that this bound is asymptotically tight up to a factor of $\frac{\ln^2 n}{\ln \ln n}$, when $s$ ($= |L|$) is a constant. However, when $s$ is a constant, we believe  that $|\mathcal{F}| \in \Theta(n^s)$. 

Consider the following special case for a fractional $L$-intersecting family $\mathcal{F}$, where $L=\{\frac{1}{2}\}$. We call such a family a bisection-closed family (see definition in Section \ref{sec:singleton}). 
\begin{conjecture}
If  $\mathcal{F}$ is a bisection-closed  family, then $|\mathcal{F}| \leq cn$, where $c>0$ is a constant. 
\end{conjecture}   
We have not been able  to find  an example of  a bisection-closed family of size $2n$ or more. 
}
\niranjan{
\rogers{The problem of determining a linear sized upper bound for the size of any bisection-closed family} leads us to pose the following question:
\rogers{\begin{openproblem} Suppose $0<a_1\le\cdots\le a_n$ are $n$ distinct reals. Let $\M_n(a_1,\ldots,a_n)$ denote the set of all symmetric matrices $M$ satisfying $m_{ij}\in\{a_i,a_j\}$ for $i\ne j$ and $m_{ii}=0$ for all $i$. Then, does there exist an absolute constant $c>0$ such that $rk(M)\ge cn,\textrm{\ for\ all\ }M\in\mathcal{M}_n(a_1,\ldots,a_n$)?
\end{openproblem}}}

\niranjan{\rogers{To see how this question ties in with our problem, suppose that a family $\F\subset\Pa([n])$ is a bisection closed family}, i.e., for $A, B\in\F$ and $A\ne B$ then $|A\cap B|\in\{|A|/2,|B|/2\}$. For simplicity, let us write $\F=\{A_1,\ldots,A_m\}$ and  denote $|A_i|=a_i$ where the $a_i$ are arranged in ascending order. We say $A$ bisects $B$ if $|A\cap B|=|B|/2$. For each $A\in\mathcal{F}$, let  $\you_A\in\bR^n$ where $\you_A(i)= 1$ if $i\in A$ and $-1$ if $i\not\in A$. Then note that 
\begin{eqnarray*} \langle \you_A,\you_B\rangle &=& n-2|A|\hspace{1cm}\textrm{\ if\ }A\textrm{\ bisects\ } B,\\
                                                                            &=& n-2|B|\hspace{1cm}\textrm{\ if\ }B\textrm{\ bisects\ } A,\\
                                               \parallel\you_A\parallel^2 &=& n.
                              \end{eqnarray*}
Consider the $m\times m$ matrix $M$ whose rows and columns are indexed by the members of $\F$, with $M_{A,B} =\langle \you_A,\you_B\rangle$. Then, since $M$ is a Gram matrix of vectors in $\bR^n$, it follows that \rogers{$rk(M)\le n$}. If $\mathcal{X}=\frac{1}{2}(nJ-M)$, where $J$ is the all ones matrix of order $m$, then \rogers{$rk(\mathcal{X})\le n+1$}. But note that $\mathcal{X}\in\M(a_1,\ldots,a_m)$. \rogers{So, if the answer to the aforementioned open problem is `yes', then $rk(\mathcal{X})\ge cm$}. This gives $cm\le r(\mathcal{X})\le n+1\textrm{\ which\ in\ turn\ gives\ } m\le c^{-1}(n+1)$. 
}

\niranjan{The problem of determining the maximum size of a fractional $L$-intersecting family is far from robust in the following sense. Suppose $L=\{1/2\}$ and we consider the problem of determining the size of an `$\ep$-approximately fractional $L$-intersecting family,' i.e., for any $A\neq B$ we have that at least one of $\frac{|A\cap B|}{|A|},\frac{|A\cap B|}{|B|}\in(1/2-\ep,1/2+\ep)$ for small $\ep>0$, then such families can in fact be exponentially large in size. \rogers{Let each set $A_i$ be chosen uniformly and independently at random from $\Pa([n])$. Then since each $|A_i|$ and $|A_i\cap A_j|$ are independent binomial $B(n,1/2)$ and $B(n,1/4)$ respectively, by standard Chernoff bounds (see \cite{MolloyReed}, chapter 5), it follows (by straightforward computations) that one can get such a family of cardinality at least $e^{2\ep^2 n/75}$.} In fact this same construction gives super-polynomial sized families even if $\ep=n^{-1/2+\delta}$ for any fixed $\delta>0$.}

\tapas{ Another interesting facet of the fractional intersection notion is the following extension of $l$-avoiding families \cite{MR871675,keevash2017frankl}
	\footnote{A family $\mathcal{F}$ is called $l$-avoiding if for each $A,B \in \mathcal{F}$, $|A \cap B| \neq l$ for some $l \in [n]$.}.
	A set $B$ \emph{bisects} another set $A$ if $|A \cap B| =\frac{|A|}{2} $.
	A family $\mathcal{F}$ of even subsets of $[n]$ is called \emph{fractional $(\frac{1}{2})$-avoiding} (or \emph{bisection-free}) if
	for every $A,B \in \mathcal{F}$, neither $B$ bisects $A$ nor $A$ bisects $B$ \rogers{(if we allow odd subsets in the definition of a fractional $(\frac{1}{2})$-avoiding family, then the set of all the odd-sized subsets on $[n]$ is an example of one such family)}.
	Let $\bar{\vartheta}(n)$ denote the maximum cardinality of a fractional $(\frac{1}{2})$-avoiding family on $[n]$.
	Let $A,B \subseteq [n]$ such that $|A| > \frac{2n}{3}$ and $|B| > \frac{2n}{3}$.
	It is not very hard to see that $|A \cap B| > n/3$ whereas $|A \cap ([n]\setminus B)| < n/3$.
	So, neither $A$ can bisect $B$ nor $B$ can bisect $A$.
	Therefore, if we construct a family \rogers{$\mathcal{F}=\{A \subseteq [n] | |A| > \frac{2n}{3},~|A| \mbox{ is even.}\}$}, 
	$\mathcal{F}$ is fractional $(\frac{1}{2})$-avoiding.
	Moreover, $|\mathcal{F}| = \rogers{\sum_{2\mid i,i=0}^{\frac{n}{3}-1} \binom{n}{i}} 
	> 1.88^n$, for sufficiently large $n$ (using Stirling's formula). \rogers{Let us now try to find an upper bound to the cardinality of a fractional $(\frac{1}{2})$-avoiding family.} 
	An application of a result of Frankl and R\'{o}dl \cite[Corollary 1.6]{MR871675} gives the following theorem for the cardinalities of 
	$l$-avoiding families as a corollary (see \cite[Theorem 1.1]{keevash2017frankl}).
	
	\begin{theorem}\cite{MR871675,keevash2017frankl}
		Let $\alpha,\epsilon \in (0,1)$ with $\epsilon \leq \frac{\alpha}{2}$. Let $k = \lfloor \alpha n \rfloor$ and
		$l \in [\max(0,2k-n) + \epsilon n,k-\epsilon n]$. Then any $l$-avoiding family $A \subseteq \binom{[n]}{k}$ satisfies
		$|A| \leq (1-\delta)^n \binom{n}{k}$ where $\delta=\delta(\alpha,\epsilon)> 0$.
		\label{thm:frankl1987}
	\end{theorem}
	For any fractional $(\frac{1}{2})$-avoiding family $\mathcal{F}$, 
	any $\mathcal{F}' \subseteq \mathcal{F}$ consisting of sets of cardinality $l$ is $\frac{l}{2}$-avoiding.
	So, given any fractional $(\frac{1}{2})$-avoiding family $\mathcal{F}$, split $\mathcal{F}$ into families
	$\mathcal{F}_{\leq \frac{n}{3}-1}, \mathcal{F}_{\frac{n}{3}},\ldots,$  $\mathcal{F}_{2\frac{n}{3}}, \mathcal{F}_{\geq \frac{2n}{3}+1}$.
	From Theorem \ref{thm:frankl1987}, we know that each $\mathcal{F}_{i}$ has a cardinality at most $(1-\delta_i)^n \binom{n}{i}$ for $\frac{n}{3} \leq i \leq \frac{2n}{3}$. \rogers{Let $\delta=\min(\delta_{\frac{n}{3}},\ldots, \delta_{\frac{2n}{3}})$. Then $\sum_{i=\frac{n}{3}}^{\frac{2n}{3}}|\mathcal{F}_i| \leq ((1-\delta)2)^n$.
	Further, $|\mathcal{F}_{\leq \frac{n}{3}-1}| \leq \sum_{i=0}^{\frac{n}{3}-1} \binom{n}{i}$ and
	$|\mathcal{F}_{\geq \frac{2n}{3}+1}| \leq \sum_{i=0}^{\frac{n}{3}-1} \binom{n}{i} < 2^{nH(\frac{1}{3})} < 1.89^n$,
	where $H(\nu)=-\nu \log_2 \nu -(1-\nu)\log_2 (1-\nu)$ is the binary entropy function. Thus, for sufficiently large values of $n$, $1.88^n \leq \bar{\vartheta}(n) \leq ((1-\epsilon)2)^{n}$, for some $0 < \epsilon \leq 0.06$.} 
	
}

\bibliographystyle{plain}

\begin{thebibliography}{10}

\bibitem{alon2009perturbed}
Noga Alon.
\newblock Perturbed identity matrices have high rank: Proof and applications.
\newblock {\em Combinatorics, Probability and Computing}, 18(1-2):3--15, 2009.

\bibitem{alon1991multilinear}
Noga Alon, L{\'a}szl{\'o} Babai, and Hiroshi Suzuki.
\newblock Multilinear polynomials and frankl-ray-chaudhuri-wilson type
  intersection theorems.
\newblock {\em Journal of Combinatorial Theory, Series A}, 58(2):165--180,
  1991.

\bibitem{babai1992linear}
L{\'a}szl{\'o} Babai and P{\'e}ter Frankl.
\newblock {\em Linear Algebra Methods in Combinatorics: With Applications to
  Geometry and Computer Science}.
\newblock Department of Computer Science, univ. of Chicag, 1992.

\bibitem{CODENOTTI200089}
Bruno Codenotti, Pavel Pudlák, and Giovanni Resta.
\newblock Some structural properties of low-rank matrices related to
  computational complexity.
\newblock {\em Theoretical Computer Science}, 235(1):89 -- 107, 2000.

\bibitem{fisher1940examination}
Ronald~Aylmer Fisher.
\newblock An examination of the different possible solutions of a problem in
  incomplete blocks.
\newblock {\em Annals of Human Genetics}, 10(1):52--75, 1940.

\bibitem{MR871675}
Peter Frankl and Vojt{\v{e}}ch R{\"o}dl.
\newblock Forbidden intersections.
\newblock {\em Trans. Amer. Math. Soc.}, 300(1):259--286, 1987.

\bibitem{frankl1981intersection}
Peter Frankl and Richard~M. Wilson.
\newblock Intersection theorems with geometric consequences.
\newblock {\em Combinatorica}, 1(4):357--368, 1981.

\bibitem{grolmusz2002k}
Vince Grolmusz and Benny Sudakov.
\newblock On k-wise set-intersections and k-wise hamming-distances.
\newblock {\em Journal of Combinatorial Theory, Series A}, 99(1):180--190,
  2002.

\bibitem{jukna2011extremal}
Stasys Jukna.
\newblock {\em Extremal combinatorics: with applications in computer science}.
\newblock Springer Science \& Business Media, 2011.

\bibitem{keevash2017frankl}
Peter Keevash and Eoin Long.
\newblock Frankl-r{\"o}dl-type theorems for codes and permutations.
\newblock {\em Transactions of the American Mathematical Society},
  369(2):1147--1162, 2017.

\bibitem{liu2014set}
Jiuqiang Liu and Wenbo Yang.
\newblock Set systems with restricted k-wise l-intersections modulo a prime
  number.
\newblock {\em European Journal of Combinatorics}, 36:707--719, 2014.

\bibitem{MolloyReed}
Michael Molloy and Bruce Reed.
\newblock {\em Graph colouring and the probabilistic method}, volume~23.
\newblock Springer Science \& Business Media, 2013.

\bibitem{ray1975t}
Dijen~K. Ray-Chaudhuri and Richard~M. Wilson.
\newblock On t-designs.
\newblock {\em Osaka Journal of Mathematics}, 12(3):737--744, 1975.

\bibitem{snevily2003sharp}
Hunter~S. Snevily.
\newblock A sharp bound for the number of sets that pairwise intersect at $k$
  positive values.
\newblock {\em Combinatorica}, 23(3):527--533, 2003.

\end{thebibliography}

\end{document}